\documentclass[10pt, a4paper, twoside]{amsart}
\usepackage{amsthm}
\usepackage{amsmath}
\usepackage{amssymb}
\usepackage{amscd}
\usepackage{mathtools}
\usepackage[all]{xy}
\usepackage{indentfirst}
\usepackage{comment}
\usepackage[pdftex]{graphicx}
\usepackage{tikz}
\usetikzlibrary{patterns}

\newtheorem{thm}{\bf Theorem}[section]
\newtheorem{prop}[thm]{\bf Proposition}
\newtheorem{lem}[thm]{\bf Lemma}

\newtheorem{cor}[thm]{\bf Corollary}

\newtheorem{q}[thm]{\bf Question}

\newtheorem*{thm*}{\bf Theorem}
\newtheorem*{cor*}{\bf Corollary}

\theoremstyle{definition}

\newtheorem{rem}[thm]{\it Remark}

\newtheorem*{df*}{\bf Definition}
\newtheorem*{not*}{\bf Notation}
\newtheorem*{dfs*}{\bf Definitions}
\newtheorem*{ack*}{\bf Acknowledgements}
\newtheorem*{dfrem*}{\bf Definition and Remark}

\def\P{\mathbb{P}}
\def\C{\mathbb{C}}
\def\Q{\mathbb{Q}}

\def\Z{\mathbb{Z}}

\DeclareMathOperator{\Proj}{Proj}

\DeclareMathOperator{\Ker}{Ker}

\DeclareMathOperator{\Image}{Im}

\DeclareMathOperator{\corank}{co-rank}
\DeclareMathOperator{\Coker}{Coker}
\DeclareMathOperator{\tr}{tr}
\DeclareMathOperator{\tors}{tors}

\DeclareMathOperator{\alg}{alg}

\makeindex

\subjclass[2010]{14C25, 14C30, 14C35}
\keywords{Regular homomorphisms, Abel-Jaobi map, Chow groups}
\title[\tiny A remark on a $3$-fold constructed by Colliot-Th\'el\`ene and Voisin]
{A remark on a $3$-fold constructed by Colliot-Th\'el\`ene and Voisin}
\author{Fumiaki Suzuki}
\address{Department of Mathematics, Statistics, and Computer Science, University of Illinois at Chicago}
\email{fsuzuk2@uic.edu}

\begin{document}
\maketitle

\begin{abstract}
A classical question asks whether the Abel-Jacobi map is universal among all regular homomorphisms.
In this paper, we prove that we can construct a $4$-fold which gives the negative answer in codimension $3$ 
if the generalized Bloch conjecture holds for a $3$-fold constructed by Colliot-Th\'el\`ene and Voisin
 in the context of the study of the defect of the integral Hodge conjecture in degree $4$.
\end{abstract}

\section{introduction}
For a smooth complex projective variety $X$,
we denote by $CH^{p}(X)$ the $p$-th Chow group of $X$ and by $A^{p}(X)$ 
(resp. $CH^{p}(X)_{\hom}$) 
the subgroup of cycle classes algebraically equivalent to zero
(resp. homologous to zero).
A homomorphism $\phi\colon A^{p}(X)\rightarrow A$ to an Abelian variety $A$ is called {\it regular} 
if for any smooth connected projective variety $S$ with a base point $s_{0}$ and for any codimension $p$ cycle $\Gamma$ on $S\times X$,
the composition 
\[S\rightarrow A^{p}(X) \rightarrow A, s\mapsto \phi(\Gamma_{*}(s-s_{0}))\]
is a morphism of algebraic varieties.
An important example of such homomorphisms is given as follows.
We consider the Abel-Jacobi map
$CH^{p}(X)_{\hom}\rightarrow J^{p}(X)$,
where $J^{p}(X)= H^{2p-1}(X, \C)/(H^{2p-1}(X,\Z(p)) + F^{p}H^{2p-1}(X, \C))$ is the $p$-th Griffith intermediate Jacobian.
Then the image $J^{p}_{a}(X)\subset J^{p}(X)$ of the restriction of the Abel-Jacobi map to $A^{p}(X)$ is an Abelian variety, and the induced map
\[
\psi^{p}\colon A^{p}(X)\rightarrow J^{p}_{a}(X),
\]
which we also call Abel-Jacobi,
is regular \cite{G}\cite{L}.
A classical question \cite[Section 7]{M3} asks whether the Abel-Jacobi map $\psi^{p} \colon A^{p}(X)\rightarrow J^{p}_{a}(X)$ is {\it universal} among all regular homomorphisms $\phi\colon A^{p}(X)\rightarrow A$,
that is, whether every such $\phi$ factors through $\psi^{p}$.
It is true for $p=1$ by the theory of Picard variety, and for $p=\dim X$ by the theory of Albanese variety.
It is also true for $p=2$, which is proved by Murre \cite{M1}\cite{M2} using the Merkurjev-Suslin theorem \cite{MS}.
To the author's knowledge, this question is open for $3\leq p\leq \dim X-1$.

Some progress has been made. 
As an application of the theory of the Lawson homology and the morphic cohomology, the following theorem is proved by Walker \cite{W}:
the Abel-Jacobi map $\psi^{p}$ factors as
\[
\xymatrix{
 &J(N^{p-1}H^{2p-1}(X, \Z(p)))\ar[d]^-{\pi^{p}}\\
 A^{p}(X)\ar[ur]^-{\widetilde{\psi}^{p}}\ar[r]_-{\psi^{p}}& J^{p}_{a}(X)\\
},
\]
where $J(N^{p-1}H^{2p-1}(X, \Z(p)))$ is the intermediate Jacobian for the mixed Hodge structure given by the coniveau filtration $N^{p-1}H^{2p-1}(X, \Z(p))$,
 $\pi^{p}$ is the natural map,
and $\widetilde{\psi}^{p}$ is a surjective regular homomorphism.
If the Abel-Jacobi map $\psi^{p}$ is universal, then the kernel
\[\Ker(\pi^{p})= \Coker\left(H^{2p-1}(X,\Z(p))_{\tors}\rightarrow (H^{2p-1}(X,\Z(p))/N^{p-1}H^{2p-1}(X,\Z(p)))_{\tors}\right)\]
is trivial.
In other words, the sublattice
\[
N^{p-1}H^{2p-1}(X,\Z(p))/\tors\subset H^{2p-1}(X,\Z(p))/\tors
\] 
is primitive.

We prove that the converse also holds for any $p\in \left\{3, \dim X-1\right\}$ if the Chow group $CH_{0}(X)$ of $0$-cycles is supported on a $3$-dimensional closed subset.
It is a consequence of an analogue of the Roitman theorem \cite[Theorem 3.1]{R}:

\begin{thm}\label{tA}
Let $Y$ be a smooth projective variety such that $CH_{0}(Y)$ is supported on a $3$-dimensional closed subset.
Let $p\in \left\{3, \dim Y-1\right\}$.
Then the restriction 
\[\widetilde{\psi}^{p}|_{\tors} \colon A^{p}(Y)_{\tors} \rightarrow J(N^{p-1}H^{2p-1}(Y, \Z(p)))_{\tors}\]
is an isomorphism.
Moreover the Walker map $\widetilde{\psi}^{p}$ is universal.
\end{thm}
\begin{rem}
There is a $4$-fold $Y$ such that $A^{3}(Y)$ has infinite $l$-torsion for some prime $l$ \cite{Sch}.
Therefore the assumption on $CH_{0}(Y)$ is essential.
\end{rem}


Then we prove that we can construct a $4$-fold which does not meet the primitiveness condition for $p=3$
and gives the negative answer to the universality question in codimension $3$
if the generalized Bloch conjecture holds for a $3$-fold constructed by Colliot-Th\'el\`ene and Voisin
in the context of the study of 
the defect
\[Z^{4}(X)=Hdg^{4}(X,\Z)/H^{4}_{\alg}(X,\Z(2))\]
of 
the integral Hodge conjecture in degree $4$.
The $3$-fold is constructed as an example which 
conjecturally 
gives the negative answer to their question \cite[Section 6]{CTV}
asking whether $Z^{4}(X)$ is trivial for a $3$-fold $X$ such that 
$CH_{0}(X)$ is supported on a proper closed subset.
It is constructed as follows \cite[Subsection 5.7]{CTV}.
Let $G=\Z/5$.
We fix a generator $g\in G$, then $G$ acts on $\P^{1}=\Proj\C[X, Y]$ and $\P^{3}=\Proj\C[X_{0},\cdots, X_{3}]$ by
\[
g^{*}X = X, g^{*}Y=\zeta Y, g^{*}X_{i} = \zeta^{i}X_{i}\, (i=0,\cdots, 3),
\]
where $\zeta$ is a primitive $5$-th root of unity. 
Let $ H\subset \P^{1}\times\P^{3}$ be the hypersurface defined by a very general $G$-invariant polynomial of type $(3,4)$, and $X$ be a resolution of $H/G$.
Then $X$ satisfies the following two properties \cite[Proposition 5.7]{CTV}:
\begin{enumerate}
\item[(i)] $H^{p}(X, \mathcal{O}_{X}) =0$ for all $p>0$;
\item[(ii)] there is a class $\alpha\in H^{4}(X,\Z(2))=Hdg^{4}(X, \Z)$ such that $\alpha \cdot F = 5$,
but $C\cdot F$ is even  for any curves $C\subset X$, where $F$ is the class of fibers of the morphism $X\rightarrow \P^{1}$ induced by the first projection $H\rightarrow \P^{1}$.
\end{enumerate}
The property (i) implies $CH_{0}(X)=\Z$ if the generalized Bloch conjecture \cite[Conjecture 11.23]{V2} holds for $X$.
The property (ii) implies more than $Z^{4}(X)\neq 0$.
We have $\Coker\left(H^{4}(X, \Z(2))_{\tors} \rightarrow Z^{4}(X)\right)\neq 0$,
where a non-zero element is given by the image of $\alpha$.
Equivalently,
the sublattice 
\[H^{4}_{\alg}(X,\Z(2))/\tors \subset Hdg^{4}(X,\Z)/\tors\]
 is not primitive.
Our main theorem is:

\begin{thm}\label{t}
Let $X$ be a smooth projective variety such that 
$CH_{0}(X)$ is supported on a $2$-dimensional closed subset and 
the sublattice 
\[H^{4}_{\alg}(X,\Z(2))/\tors \subset Hdg^{4}(X,\Z)/\tors\]
is not primitive.
Then there exists a smooth elliptic curve $E$ such that 
the sublattice
\[N^{2}H^{5}(X\times E,\Z(3))/\tors \subset H^{5}(X \times E,\Z(3))/\tors\]
is not primitive
and the Abel-Jacobi map $\psi^{3}\colon A^{3}(X\times E)\rightarrow J^{3}_{a}(X\times E)$ is not universal.
\end{thm}

\begin{cor}\label{c}
Let $X$ be the $3$-fold constructed by Colliot-Th\'el\`ene and Voisin.
Assume that the generalized Bloch conjecture holds for $X$.
Then there exists a smooth elliptic curve $E$ such that the sublattice
\[N^{2}H^{5}(X\times E,\Z(3))/\tors \subset H^{5}(X\times E,\Z(3))/\tors\]
is not primitive
and the Abel-Jacobi map $\psi^{3}\colon A^{3}(X\times E)\rightarrow J^{3}_{a}(X\times E)$ is not universal.
\end{cor}

This paper is organized as follows.
In Section $2$, we study regular homomorphisms on the torsion subgroup $A^{p}(X)_{\tors}$.
Then we prove Theorem \ref{tA} and its corollary.
In Section $3$, we prove a proposition on non-algebraic integral Hodge classes of Koll\'ar type
and torsion elements in the Abel-Jacobi kernel.
Then we conclude Theorem \ref{t} and its ``homology counterpart''.
In Section $4$, we discuss stably birational invariants related to our problem.
In Appendix $A$, we prove the Roitman theorem for the Walker maps by using the formalism of decomposition of the diagonal.

Throughout this paper, the base field is the field of complex numbers $\C$.

\begin{not*}
For a smooth projective variety $X$, we denote by $N^{i}H^{j}(X,\Z(k))$ the coniveau filtration on $H^{j}(X, \Z(k))$. 
Recall that it is defined as
\[
N^{i}H^{j}(X,\Z(k)) = \Ker\left(H^{j}(X,\Z(k))\rightarrow \varinjlim_{Z\subset X} H^{j}(X-Z, \Z(k))\right),
\]
where $Z\subset X$ runs through all codimension $\geq i$ closed subsets of $X$.
For an Abelian group $G$ and a prime number $l$, we denote by $G_{l\mathchar`-\tors}$ the subgroup of $l$-primary torsion elements of $G$.
\end{not*}

\begin{ack*}
The author wishes to thank his advisor Lawrence Ein for constant support and warm encouragement.
He wishes to thank Henri Gillet for fruitful discussions.
He also thanks Jean-Louis Colliot-Th\'el\`ene and Mark Walker for helpful comments.
Finally, he thanks the anonymous referee for valuable suggestions which improve the exposition of this paper.
\end{ack*}

\section{Regular homomorphisms on the torsion subgroup $A^{p}(X)_{\tors}$}

\begin{lem}\label{p1}
Let $X$ be a smooth projective variety and $\phi\colon A^{p}(X)\rightarrow A$ be a surjective regular homomorphism.
Assume that the restriction $\phi|_{\tors}\colon A^{p}(X)_{\tors}\rightarrow A_{\tors}$ is an isomorphism.
Then $\phi$ is universal.
\end{lem}
\begin{proof}
First we prove the existence of a universal regular homomorphism $\phi_{0}\colon A^{p}(X)\rightarrow A_{0}$.
By Saito's criterion \cite[Theorem 2.2]{S} (see also \cite[Proposition 2.1]{M2}),
it is enough to prove $\dim B\leq \dim A$ for any surjective regular homomorphism $\psi\colon A^{p}(X)\rightarrow B$.
Such a homomorphism $\psi$ restricts to a surjection $\psi|_{\tors}\colon A^{p}(X)_{\tors}\rightarrow B_{\tors}$.
Indeed, by \cite[Proposition 1.2]{S} (see also \cite[Lemma 1.6.2]{M2}), there exists an Abelian variety $C$ and $\Gamma \in CH^{p}(C\times X)$ such that the composition
\[
C\rightarrow A^{p}(X) \rightarrow B, s\mapsto \psi(\Gamma_{*}(s-s_{0}))
\]
is an isogeny,
and the restriction $C_{\tors}\rightarrow B_{\tors}$ is a surjection.
By assumption, we have $A^{p}(X)_{\tors}\cong A_{\tors}$.
Then we have
\[
\dim B = \frac{1}{2}\corank B_{\tors} \leq \frac{1}{2}\corank A_{\tors} = \dim A.
\]
The existence follows.

The map $\phi_{0}$ is surjective since the image of a regular homomorphism is an Abelian variety \cite[Lemma 1.6.2]{M2}.
Thus $\phi_{0}$ restricts to a surjection $\phi_{0}|_{\tors}\colon A^{p}(X)_{\tors}\rightarrow (A_{0})_{\tors}$.
The induced map $A_{0}\rightarrow A$ is surjective 
 and restricts to an isomorphism $(A_{0})_{\tors}\cong A_{\tors}$,
therefore it is an isomorphism.
The proof is done.
\end{proof}

We review the Bloch-Ogus theory on the coniveau spectral sequence \cite{BO}.
For a smooth projective variety $X$,
we define $\mathcal{H}^{q}(\Z(r))$ to be the Zariski sheaf on $X$ associated to the presheaf $U\mapsto H^{q}(U,\Z(r))$.
Then the $E_{2}$ term of the coniveau spectral sequence is given by
\[
E^{p,q}_{2} = H^{p}(X, \mathcal{H}^{q}(\Z(r))) \Rightarrow N^{\bullet}H^{p+q}(X, \Z(r)),
\]
and we have $E^{p,q}_{2}=0$ if $p>q$ \cite[Corollary 6.3]{BO}.
We also have
\[E^{p,q}_{2}= 0 \text{ if }(p,q)\not \in [0,\dim X]\times [0, \dim X].\]

Let $f^{p}\colon H^{p-1}(X, \mathcal{H}^{p}(\Z(p)))\rightarrow H^{2p-1}(X,\Z(p))$ be the edge homomorphism.

\begin{lem}\label{t1}
There is a short exact sequence:
\begin{eqnarray*}
0&\rightarrow &H^{p-1}(X, \mathcal{K}_{p})\otimes \Q_{l}/\Z_{l}\\
&\rightarrow &\Ker\left(f^{p}\otimes \Q_{l}/\Z_{l}\colon H^{p-1}(X, \mathcal{H}^{p}(\Z(p)))\otimes \Q_{l}/\Z_{l} \rightarrow N^{p-1}H^{2p-1}(X, \Z(p))\otimes \Q_{l}/\Z_{l}\right)\\
&\rightarrow &\Ker\left(\widetilde{\psi}^{p}|_{l\mathchar`-\tors}\colon A^{p}(X)_{l\mathchar`-\tors}\rightarrow J(N^{p-1}H^{2p-1}(X, \Z(p)))_{l\mathchar`-\tors} \right) \rightarrow 0
\end{eqnarray*}
for any smooth projective variety $X$ and any prime number $l$, where $\mathcal{K}_{p}$ is the Zariski sheaf on $X$ associated to the Quillen K-theory.
\end{lem}
\begin{proof}
We use the Bloch map $\lambda^{p}_{l} \colon CH^{p}(X)_{l\mathchar`-\tors} \rightarrow H^{2p-1}(X, \Q_{l}/\Z_{l}(p))$ \cite{B1} (see also \cite{C}).
By the construction of the Bloch map and \cite[Theorem 5.1]{Ma}, we have a commutative diagram with exact rows:
\[
\xymatrixcolsep{1pc}
\xymatrix{
0 \ar[r] & H^{p-1}(X, \mathcal{K}_{p})\otimes \Q_{l}/\Z_{l} \ar[r] \ar[d] & H^{p-1}(X, \mathcal{H}^{p}(\Z(p)))\otimes \Q_{l}/\Z_{l} \ar[r]\ar[d]^{f^{p}\otimes \Q_{l}/\Z_{l}} &A^{p}(X)_{l\mathchar`-\tors} \ar[r]\ar[d]^{-\lambda^{p}_{l}}& 0\\
  & 0 \ar[r] &H^{2p-1}(X, \Z(p))\otimes \Q_{l}/\Z_{l} \ar[r] & H^{2p-1}(X, \Q_{l}/\Z_{l}(p)) &\\
}.
 \]
We prove that it induces another commutative diagram:
\[
\xymatrixcolsep{1pc}
\xymatrix{
0 \ar[r] & H^{p-1}(X, \mathcal{K}_{p})\otimes \Q_{l}/\Z_{l} \ar[r] \ar[d] & H^{p-1}(X, \mathcal{H}^{p}(\Z(p)))\otimes \Q_{l}/\Z_{l} \ar[r]\ar[d]^{f^{p}\otimes \Q_{l}/\Z_{l}} &A^{p}(X)_{l\mathchar`-\tors} \ar[r]\ar[d]^{-\widetilde{\lambda}^{p}_{l}}& 0\\
& 0 \ar[r] & N^{p-1}H^{2p-1}(X, \Z(p))\otimes \Q_{l}/\Z_{l} \ar@{=}[r] & N^{p-1}H^{2p-1}(X, \Z(p))\otimes \Q_{l}/\Z_{l} \ar[r]&0\\
}.
 \]
 It is enough to prove the image of $H^{p-1}(X, \mathcal{K}_{p})\otimes \Q_{l}/\Z_{l}$ in $N^{p-1}H^{2p-1}(X, \Z(p))\otimes \Q_{l}/\Z_{l}$ is zero.
This follows by observing that $H^{p-1}(X, \mathcal{K}_{p})\otimes \Q_{l}/\Z_{l}$ is divisible and
\begin{eqnarray*}
&&\Ker\left(N^{p-1}H^{2p-1}(X, \Z(p))\otimes \Q_{l}/\Z_{l}\rightarrow H^{2p-1}(X, \Z(p))\otimes \Q_{l}/\Z_{l}\right) \\
&&= \Coker\left(H^{2p-1}(X,\Z(p))_{l\mathchar`-\tors}\rightarrow (H^{2p-1}(X, \Z(p))/N^{p-1}H^{2p-1}(X, \Z(p)))_{l\mathchar`-\tors}\right)
\end{eqnarray*}
is finite.
We prove that $\widetilde{\lambda}^{p}$ coincides with the restriction $\widetilde{\psi}^{p}|_{l\mathchar`-\tors}$.
In commutative triangles
\[
\xymatrixcolsep{5pc}
\xymatrix{
 &N^{p-1}H^{2p-1}(X,\Z(p))\otimes \Q_{l}/\Z_{l}\ar[d]\\
 A^{p}(X)_{l\mathchar`-\tors}\ar[ur]^-{\widetilde{\lambda}^{p}_{l}\,(\text{resp. }\widetilde{\psi}^{p}|_{l\mathchar`-\tors})}\ar[r]_-{\lambda^{p}_{l}\,(\text{resp. }\psi^{p}|_{l\mathchar`-\tors})}& H^{2p-1}(X,\Z(p))\otimes \Q_{l}/\Z_{l}\\
},
\]
$\widetilde{\lambda}^{p}_{l}$ (resp. $\widetilde{\psi}^{p}|_{l\mathchar`-\tors}$) is  the unique lift of $\lambda^{p}_{l}$ (resp. $\psi^{p}|_{l\mathchar`-\tors}$) since $A^{p}(X)_{l\mathchar`-\tors}$ is divisible \cite[Lemma 7.10]{BO}.
Therefore it is enough to prove that $\lambda^{p}_{l}$ coincides with $\psi^{p}|_{l\mathchar`-\tors}$.
This follows from \cite[Proposition 3.7]{B1}.
The proof is done by the snake lemma.
\end{proof}

\begin{proof}[Proof of Theorem \ref{tA}]
The second statement follows from the first one by Lemma \ref{p1}.

We prove that $\widetilde{\psi}^{3}|_{\tors}$ is an isomorphism.
By Lemma \ref{t1}, it is enough to prove that 
\[\Ker(f^{3}) = \Image\left(H^{0}(Y, \mathcal{H}^{4}(\Z(3)))\rightarrow H^{2}(Y, \mathcal{H}^{3}(\Z(3)))\right)\] is torsion.
The group $H^{0}(Y,\mathcal{H}^{4}(\Z(3)))$ is torsion by \cite[Proposition 3.3 (i)]{CTV} (it is actually zero as a consequence of the Bloch-Kato conjecture, see \cite[Theorem 3.1]{CTV}), 
so the result follows.

Let $d=\dim Y$.
We prove that $\widetilde{\psi}^{d-1}|_{\tors}$ is an isomorphism.
By Lemma \ref{t1}, it is enough to prove that 
\[\Ker(f^{d-1}) = \Image\left(H^{d-4}(Y, \mathcal{H}^{d}(\Z(d-1)))\rightarrow H^{d-2}(Y, \mathcal{H}^{d-1}(\Z(d-1)))\right)\] 
is torsion.
The group $H^{d-4}(Y,\mathcal{H}^{d}(\Z(d-1)))$ is torsion by \cite[Proposition 3.3 (ii)]{CTV}, so the result follows.
\end{proof}

\begin{cor}\label{cA}
Under the assumptions of Theorem \ref{tA},
the following are equivalent:
\begin{enumerate}
\item[(i)] the Abel-Jacobi map $\psi^{p}$ is universal;
\item[(ii)] the sublattice $N^{p-1}H^{2p-1}(Y,\Z(p))/\tors \subset H^{2p-1}(Y,\Z(p))/\tors$
is primitive;
\item[(iii)] the restriction $\psi^{p}|_{\tors}\colon A^{p}(Y)_{\tors}\rightarrow J^{p}_{a}(Y)_{\tors}$
is an isomorphism.
\end{enumerate}
\end{cor}
\begin{proof}
It is enough to prove that (ii) and (iii) are equivalent.
By Theorem \ref{tA}, we have an isomorphism
\[
\Ker\left(\pi^{p})\cong\Ker(\psi^{p}|_{\tors}\colon A^{p}(Y)_{\tors}\rightarrow J^{p}_{a}(Y)_{\tors}\right).
\]
The result follows.
\end{proof}

\section{Non-algebraic integral Hodge classes of Koll\'ar type and torsion elements in the Abel-Jacobi kernel}
Inspired by the work of Soul\'e and Voisin \cite{SV}, we prove:

\begin{prop}\label{tB}
Let $X$ be a smooth projective variety such that 
the sublattice
\[H^{2p}_{\alg}(X,\Z(p))/\tors \subset Hdg^{2p}(X,\Z)/\tors\]
is not primitive.
Then there exists a smooth elliptic curve $E$ such that
the restriction
\[
\psi^{p+1}|_{\tors}\colon A^{p+1}(X\times E)_{\tors}\rightarrow J^{p+1}_{a}(X\times E)_{\tors}
\]
is not an isomorphism.
\end{prop}
\begin{rem}
The assumption of Proposition \ref{tB} for $p=2$ is satisfied by Koll\'ar's example \cite[p.134, Lemma]{BCC} (see also \cite[Section 2]{SV}).
It is a very general hypersurface in $\P^{4}$ of degree $l^{3}$ for a prime number $l\geq 5$.
When it contains a certain smooth degree $l$ curve, the same conclusion follows from \cite[Theorem 4]{SV}.
The details are given in \cite[Section 4]{SV}.
\end{rem}
\begin{proof}[Proof of Proposition \ref{tB}]
We define
\[
\overline{Z}^{2p}(X)=\Coker\left(H^{2p}(X,\Z(p))_{\tors}\rightarrow Hdg^{2p}(X,\Z)/H^{2p}_{\alg}(X,\Z(p)) \right).
\]
Then we have the following exact sequence:
\[
0\rightarrow 
\overline{Z}^{2p}(X)_{\tors}
\rightarrow H^{2p}_{\alg}(X, \Z(p))\otimes \Q/\Z \rightarrow Hdg^{2p}(X, \Z)\otimes \Q/\Z.
\]
We have $\overline{Z}^{2p}(X)_{\tors}\neq 0$ by the assumption.
Let $\alpha \in 
\overline{Z}^{2p}(X)_{\tors}$
be a non-trivial element;
we use the same notation for its image in $H^{2p}_{\alg}(X,\Z(p))\otimes \Q/\Z$.
Let $\widetilde{\alpha}\in CH^{p}(X)\otimes \Q/\Z$ be an element which maps to $\alpha$ via the surjection 
\[cl^{p}\otimes \Q/\Z \colon CH^{p}(X)\otimes \Q/\Z \rightarrow H^{2p}_{\alg}(X, \Z(p))\otimes \Q/\Z.\]
Let $k\subset \C$ be an algebraically closed field such that $\tr.\deg_{\Q}k <\infty$ and both $X$ and $\widetilde{\alpha}$ are defined over $k$.
Let $E$ be a smooth elliptic curve such that $j(E)\not\in k$.
We fix one component $\Q/\Z$ of $CH^{1}(E)_{\tors}=(\Q/\Z)^{2}$,
and we identify $\widetilde{\alpha}$ with an element in $CH^{p}(X)\otimes CH^{1}(E)_{\tors}$.
By the Schoen theorem \cite[Theorem 0.2]{Sch}, the image $\beta$ of $\widetilde{\alpha}$ by the exterior product map
\[
CH^{p}(X)\otimes CH^{1}(E)_{\tors}\rightarrow CH^{p+1}(X\times E)
\]
is non-zero.
Then $\beta \in A^{p+1}(X\times E)_{\tors}$.
We prove 
\[\beta \in \Ker\left(\psi^{p+1}\colon A^{p+1}(X\times E)\rightarrow J^{p+1}_{a}(X\times E)\right).\]
It is enough to prove that $\beta$ is in the kernel of the cycle class map of the Deligne cohomology:
\[
cl_{\mathcal{D}}^{p+1}\colon CH^{p+1}(X\times E)\rightarrow H^{2p+2}_{\mathcal{D}}(X\times E, \Z(p+1)).
\]
The composition of $cl_{\mathcal{D}}^{p+1}$ with the exterior product map factors through
\[
cl_{\mathcal{D}}^{p}\otimes cl_{\mathcal{D}}^{1}\colon CH^{p}(X)\otimes CH^{1}(E)_{\tors}\rightarrow H^{2p}_{\mathcal{D}}(X, \Z(p))\otimes H^{2}_{\mathcal{D}}(E, \Z(1))_{\tors}.
\]
Now it is enough to prove that $\widetilde{\alpha}$ is in the kernel of this map.
Since we have an isomorphism
\[
H^{2p}_{\mathcal{D}}(X, \Z(p))\otimes H^{2}_{\mathcal{D}}(E, \Z(1))_{\tors}\cong Hdg^{2p}(X, \Z)\otimes H^{2}_{\mathcal{D}}(E, \Z(1))_{\tors},
\]
the proof is done by the choice of $\widetilde{\alpha}$.
\end{proof}

\begin{proof}[Proof of Theorem \ref{t}]
For any smooth projective curve $E$, the group $CH_{0}(X\times E)$ is supported on a $3$-dimensional closed subset.
The proof is done by applying Corollary \ref{cA} for $p=3$ to $Y=X\times E$ and Proposition \ref{tB} for $p=2$.
\end{proof}

The same arguments yield a ``homology counterpart'' of Theorem \ref{t}:

\begin{thm}
Let $X$ be a smooth projective variety such that
$CH_{0}(X)$ is supported on a $2$-dimensional closed subset
and 
the sublattice
\[
H_{2,\alg}(X,\Z(1))/\tors \subset Hdg_{2}(X,\Z)/\tors
\]
is not primitive.
Then there exists a smooth elliptic curve $E$ such that 
the sublattice
\[
N_{2}H_{3}(X\times E,\Z(1))/\tors \subset H_{3}(X\times E,\Z(1))/\tors
\]
is not primitive
and the Abel-Jacobi map $\psi_{1}\colon A_{1}(X\times E)\rightarrow J_{1,a}(X\times E)$ is not universal.
\end{thm}

\section{Stably birational invariants}

\begin{lem}\label{l1}
The groups $\Ker(\psi^{3}|_{\tors})$, $\Ker(\psi^{d-1}|_{\tors})$, $\Ker(\pi^{3})$, and $\Ker(\pi^{d-1})$, where $d=\dim X$,
are stably birational invariants of smooth projective varieties $X$.
\end{lem}
\begin{rem}
A related result is proved by Voisin \cite[Lemma 2.2]{V3}.
\end{rem}
\begin{proof}[Proof of Lemma \ref{l1}]
For each group, it is enough to check 
\begin{enumerate}
\item[(i)] the invariance under the blow-up along a smooth subvariety;
\item[(ii)] the invariance under taking the product with $\P^{n}$.
\end{enumerate}
By the formulas under these operations for the Chow groups and the Deligne cohomology groups (resp. the coniveau spectral sequence and the integral cohomology groups)
and by their compatibility with the cycle class maps (resp. the differentials and the edge homomorphisms),
(i) and (ii) are reduced to the triviality of the groups $\Ker(\psi^{i}|_{\tors})$ and $\Ker(\pi^{i})$ for $i\leq 2$ and $i=\dim Y$ on a smooth projective variety $Y$.
The triviality of $\Ker(\psi^{2}|_{\tors})$ (resp. $\Ker(\psi^{\dim Y}|_{\tors})$) follows from 
the Roitman theorem for codimension $2$-cycles due to Murre 
\cite[Theorem 10.3]{M2} 
(resp. 
the Roitman theorem 
\cite[Theorem 3.1]{R}).
The triviality of $\Ker(\pi^{2})$ (resp. $\Ker(\pi^{\dim Y})$) follows from the universality of $\psi^{2}$ (resp. $\psi^{\dim Y}$).
The rest is clear.
The proof is done.
\end{proof}

\begin{cor}
Let $X$ be a smooth projective stably rational variety.
Let $p\in\left\{3, \dim X-1\right\}$.
Then $\Ker(\psi^{p}|_{\tors})=\Ker(\pi^{p})=0$.
Therefore the Abel-Jacobi map $\psi^{p}$ is universal and the sublattice 
\[N^{p-1}H^{2p-1}(X,\Z(p))/\tors \subset H^{2p-1}(X,\Z(p))/\tors\]
is primitive.
\end{cor}

For a smooth projective variety $X$, let
$Z^{2p}(X)=Hdg^{2p}(X,\Z)/H^{2p}_{\alg}(X, \Z(p))$ be the defect of the integral Hodge conjecture in degree $2p$.
We define
\[\overline{Z}^{2p}(X)=\Coker\left(H^{2p}(X, \Z(p))_{\tors} \rightarrow Z^{2p}(X)\right).\]
Then $\overline{Z}^{2p}(X)_{\tors}=0$ if and only if the sublattice
\[
H^{2p}_{\alg}(X,\Z(p))/\tors\subset Hdg^{2p}(X,\Z)/\tors 
\]
is primitive.

\begin{lem}\label{l2}
The groups $\overline{Z}^{4}(X)$ and $\overline{Z}^{2d-2}(X)$, where $d=\dim X$, are stably birational invariants of smooth projective varieties $X$.
\end{lem}
\begin{rem}
The groups $Z^{4}(X)$ and $Z^{2d-2}(X)$, where $d=\dim X$, are stably birational invariants of smooth projective varieties $X$ \cite{V1}\cite{V4}
and related to the unramified cohomology groups \cite{CTV}.
\end{rem}
\begin{proof}[Proof of Lemma \ref{l2}]
The proof is reduced to the triviality of the groups $\overline{Z}^{2}(Y)$ and $\overline{Z}^{2\dim Y}(Y)$ on a smooth projective variety $Y$.
The triviality of $\overline{Z}^{2}(Y)$ follows from the Lefschetz $(1,1)$-theorem. 
The triviality of $\overline{Z}^{2\dim Y}(Y)$ is clear.
The proof is done.
\end{proof}

We recall the following question (see \cite[Subsection 5.6]{CTV}):

\begin{q}
Let $X$ be a smooth projective rationally connected variety.
Then is the group $\overline{Z}^{4}(X)$ trivial?
Equivalently, is the sublattice $H^{4}_{\alg}(X,\Z(2))/\tors \subset Hdg^{4}(X,\Z)/\tors$ primitive?
\end{q}

The negative answer to this question would provide us with another example to which we can apply Theorem \ref{t}.
We note that there is a unirational $6$-fold $X$ constructed by Colliot-Th\'el\`ene and Ojanguren \cite{CO} with $Z^{4}(X)$ later proved to be non-zero \cite[Theorem 1.3]{CTV}.
The $6$-fold $X$ is a smooth model of a quadric bundle $Y$ over $\P^{3}$.
It is hard to analyze $\overline{Z}^{4}(X)$ while the construction of $Y$ is explicit.

\begin{appendix}

\section{The Roitman theorem for the Walker maps and decomposition of the diagonal}
A smooth projective variety with $CH_{0}(X)$ supported on a proper closed subset admits a decomposition of the diagonal due to Bloch \cite{B2} and Bloch-Srinivas \cite{BS}.
This result is generalized by Paranjape \cite{P} and Laterveer \cite{La}.
We follow Laterveer's formulation here.
Let $X$ be a smooth projective variety of dimension $d$.
For non-negative integers $r$ and $s$, we consider the following condition:
$CH_{i}(X)_{\Q}$ is supported on an $(i+r)$-dimensional closed subset for any $0\leq i\leq s$.
We call this condition $L_{r,s}$.
Assume that $L_{r,s}$ holds for $X$.
Then $X$ admits a generalized decomposition of the diagonal \cite[Theorem 1.7]{La} (see also \cite[Proposition 6.1]{P}):
there exist closed subsets $V_{0}, \cdots, V_{s}$ and $W_{0}, \cdots, W_{s+1}$ of $X$ with
$\dim V_{j}\leq j+r$ $(j=0,\cdots, s)$ and $\dim W_{j}\leq d-j$ $(j=0, \cdots, s+1)$
such that we have a decomposition
\[
\Delta_{X}= \Delta_{0}+ \cdots + \Delta_{s} + \Delta^{s+1}
\]
in $CH^{d}(X\times X)_{\Q}$,
where $\Delta_{j}$ is supported on $V_{j}\times W_{j}$ $(j=0,\cdots, s)$ and $\Delta^{s+1}$ is supported on $X\times W_{s+1}$.

For a smooth projective variety $Y$,
let $E^{p,q}_{2}(Y)= H^{p}(Y, \mathcal{H}^{q}(\Z))$.
For the action of correspondences on the coniveau spectral sequence,
we refer the reader to \cite[Appendice A]{CTV}.

\begin{lem}\label{l}
Let $X$ be a smooth projective variety of dimension $d$ such that $L_{r,s}$ holds for $X$.
Then $E^{p,q}_{2}(X)$ is torsion if $p+r <q$ and $p<s+1$, or if $p+r <q$ and $q>d-s-1$.
\end{lem}

\begin{tikzpicture}[domain=0:4]
\draw[thick, ->] (0,0)--(4.2,0) node[right] {$p$} ; 
\draw[thick, ->] (0,0)--(0,2.7) node[above] {$q$};
\path[pattern=north east lines] (0,0.5) -- (0,2.5) to (1.5, 2.5) -- (1.5,1.25) -- (0,0.5);
\draw (1.5,0) node[above right] {$s+1$};
\draw plot(\x, {0.5*\x}) node[right] {$q=p$};
\draw[dashed] plot(\x, {0.5*\x+0.5}) node[right] {$q=p+r$};
\draw[dashed] (1.5,0)--(1.5,2.5);
\end{tikzpicture}
\begin{tikzpicture}[domain=0:4]
\draw[thick, ->] (0,0)--(4.2,0) node[right] {$p$}; 
\draw[thick, ->] (0,0)--(0,2.7) node[above] {$q$}; 
\path[pattern=north east lines] (0,1.75) -- (0,2.5) to (4, 2.5) -- (2.5,1.75) -- (0,1.75);
\draw (0,1.75) node[left] {$d-s-1$};
\draw plot(\x, {0.5*\x}) node[right] {$q=p$};
\draw[dashed] plot(\x, {0.5*\x+0.5}) node[right] {$q=p+r$};
\draw[dashed] (0,1.75)--(4, 1.75);
\end{tikzpicture}

\begin{rem}
The case $s=0$ is \cite[Proposition 3.3 (i)(ii)]{CTV}.
\end{rem}
\begin{proof}[Proof of Lemma \ref{l}]
We may assume that the inequalities about the dimensions of $V_{j}, W_{j}$ are equal.
Let $N$ be a positive integer such that
\[
N\Delta_{X} = N\Delta_{0} + \cdots + N\Delta_{s} + N \Delta^{s+1} \in CH^{d}(X\times X). 
\]
Let $\widetilde{V}_{j} (j=0,\cdots, s)$ and $\widetilde{W}_{j}(j=0,\cdots, s+1)$ be resolutions of $V_{j}$ and $W_{j}$,
and $\widetilde{\Delta}_{j}$ be $d$-cycles on $\widetilde{V}_{j}\times \widetilde{W}_{j}$ pushed forward to $c_{j} \Delta_{j}$ for some positive integer $c_{j}$.
We may assume that $c_{0} = \cdots = c_{s+1}$.
Let $N' = N\cdot c_{0}$.
We prove $N'\cdot E^{p,q}_{2}(X) =0$ if $p+r <q$ and $p<s+1$, or if $p+r <q$ and $q>d-s-1$.

For $0\leq j \leq s$, we prove that
\begin{enumerate}
\item[(i)] $(N'\Delta_{j})_{*} =0$ if $(p, q) \not\in [j, j+r]\times [j, j+r]$;
\item[(ii)] $(N'\Delta_{j})^{*} =0$ if $(p, q) \not\in [d-j-r, d-j]\times [d-j-r, d-j]$.
\end{enumerate}
We have a commutative diagram
\[
\xymatrix{
E^{p,q}_{2}(X)\ar[d] \ar[r]^-{(N'\Delta_{j})_{*}}& E^{p,q}_{2}(X) \\
E^{p,q}_{2}(\widetilde{V}_{j}) \ar[d] & E^{p -j, q -j}_{2}(\widetilde{W}_{j}) \ar[u]\\
E^{p,q}_{2}(\widetilde{V}_{j}\times\widetilde{W}_{j}) \ar[r]^-{\cup(N\widetilde{\Delta}_{j})} & E^{p+r,q+r}_{2}(\widetilde{V}_{j}\times \widetilde{W}_{j})\ar[u] \\
}.
\]
To prove (i), it is enough to observe that $E^{p,q}_{2}(\widetilde{V}_{j})=0$ if $p>j+r$ or $q>j+r$, and
$E^{p-j,q-j}_{2}(\widetilde{W}_{j}) =0$ if $p<j$ or $q <j$.
Similarly, we have a commutative diagram
\[
\xymatrix{
E^{p,q}_{2}(X)\ar[d] \ar[r]^-{(N'\Delta_{j})^{*}}& E^{p,q}_{2}(X) \\
E^{p,q}_{2}(\widetilde{W}_{j}) \ar[d] & E^{p+r-d+j, q+r-d+j}_{2}(\widetilde{V}_{j}) \ar[u]\\
E^{p,q}_{2}(\widetilde{V}_{j}\times \widetilde{W}_{j}) \ar[r]^-{\cup (N\widetilde{\Delta}_{j})} & E^{p+r,q+r}_{2}(\widetilde{V}_{j}\times \widetilde{W}_{j})\ar[u] \\
}.
\]
To prove (ii), it is enough to observe that $E^{p,q}_{2}(\widetilde{W}_{j})=0$ if  $p>d-j$ or $q>d-j$,
and $E^{p+r-d+j, q+r-d+j}_{2}(\widetilde{V}_{j}) =0$ if $p<d-r-j$ or $q< d-r-j$.

For $\Delta^{s+1}$, we prove that
\begin{enumerate}
\item[(iii)] $(N' \Delta^{s+1})_{*} =0$ if $(p,q) \not\in [s+1, d]\times [s+1, d]$;
\item[(iv)] $(N'\Delta^{s+1})^{*} =0$ if $(p,q) \not\in [0, d-s-1]\times[0, d-s-1]$.
\end{enumerate}
We have a commutative diagram
\[
\xymatrix{
E^{p,q}_{2}(X) \ar[dd] \ar[r]^-{(N'\Delta^{s+1})_{*}} & E^{p, q}_{2}(X)\\
& E^{p -s-1, q -s-1}_{2}(\widetilde{W}_{s+1}) \ar[u]\\
E^{p,q}_{2}(X\times\widetilde{W}_{s+1}) \ar[r]^-{\cup(N\widetilde{\Delta}^{s+1})} & E^{p+d-s-1, q+d-s-1}_{2}(X\times \widetilde{W}_{s+1})\ar[u] \\
}.
\]
To prove (iii), it is enough to observe that $E^{p-s-1, q-s-1}_{2}(\widetilde{W}_{s+1}) = 0$ if $p<s+1$ or $q<s+1$.
Similarly, we have a commutative diagram
\[
\xymatrix{
E^{p,q}_{2}(X) \ar[d] \ar[r]^-{(N'\Delta^{s+1})^{*}} & E^{p, q}_{2}(X)\\
E^{p,q}_{2} (\widetilde{W}_{s+1})\ar[d] & \\
E^{p,q}_{2}(X\times \widetilde{W}_{s+1}) \ar[r]^-{\cup (N\widetilde{\Delta}^{s+1})} & E^{p+d-s-1, q+d-s-1}_{2}(X\times \widetilde{W}_{s+1})\ar[uu] \\
}.
\]
To prove (iv), it is enough to observe that
$E^{p,q}(\widetilde{W}_{s+1}) = 0$ if $p>d-s-1$ or $q>d-s-1$.

The proof is done by (i), (ii), (iii) and (iv).
\end{proof}

\begin{thm}\label{tAA}
Let $X$ be a smooth projective variety of dimension $d$ such that $L_{3,s}$ holds for $X$.
Let $p\in [3, s+3]\cup [d-s-1, d-1]$.
Then the restriction 
\[\widetilde{\psi}^{p}|_{\tors}\colon A^{p}(X)_{\tors}\rightarrow J(N^{p-1}H^{2p-1}(X,\Z(p)))_{\tors}\] 
is an isomorphism. 
Moreover the Walker map $\widetilde{\psi}^{p}$ is universal.
\end{thm}
\begin{rem}
The case $s=0$ is Theorem \ref{tA}.
\end{rem}
\begin{proof}[Proof of Theorem \ref{tAA}]
The second statement follows from the first one by Lemma \ref{p1}.

We prove that the restriction $\widetilde{\psi}^{p}|_{\tors}$ is an isomorphism.
By Lemma \ref{t1}, it is enough to prove that $\Ker(f^{p})$ is torsion.
By Lemma \ref{l}, the groups 
\[E^{p-3, p+1}_{2}(X),\cdots, E^{0,2p-2}_{2}(X)\]
 are torsion, so the result follows.
\end{proof}

\end{appendix}


\begin{thebibliography}{99}
\bibitem{BCC} Ballico, E., Catanese, F., Ciliberto, C. (eds): {\it Classification of irregular varieties}, Lecture Notes in Mathematics, 1515, Springer-Verlag, Berlin, 1992. 
\bibitem{B1}Bloch, S.: {\it Torsion algebraic cycles and a theorem of Roitman}. Compositio Math. 39 (1979), no. 1, 107--127
\bibitem{B2}Bloch, S: {\it On an argument of Mumford in the theory of algebraic cycles}, in {\it Journ\'ees de G\'eometrie Alg\'ebrique d'Angers, Juillet 1979/Algebraic Geometry, Angers, 1979}, 217--221, Sijthoff \&\ Noordhoff, Alphen aan den Rijn.
\bibitem{BO}Bloch, S., Ogus, A.: {\it Gersten's conjecture and the homology of schemes}, Ann. Sci. \'Ecole Norm. Sup. (4) {\bf 7} (1974), 181--201 (1975).
\bibitem{BS}Bloch, S., Srinivas, V.: {\it Remarks on correspondences and algebraic cycles}, Amer. J. Math. {\bf 105} (1983), no.~5, 1235--1253.
\bibitem{C}Colliot-Th\'el\`ene, J. -L.: {\it Cycles alg\'ebriques de torsion et $K$-th\'eorie alg\'ebrique}, in {\it Arithmetic algebraic geometry (Trento, 1991)}, 1--49, Lecture Notes in Math., 1553, Springer, Berlin.
\bibitem{CO}Colliot-Th\'el\`ene, J. -L., Ojanguren, M.: {\it Vari\'{e}t\'{e}s unirationnelles non rationnelles: au-del\`a de l'exemple d'Artin et Mumford}, Invent. Math. {\bf 97} (1989), no.~1, 141--158.
\bibitem{CTV}Colliot-Th\'el\`ene, J. -L., Voisin, C.: {\it Cohomologie non ramifi\'ee et conjecture de Hodge enti\`ere}, Duke Math. J. {\bf 161} (2012), no.~5, 735--801.
\bibitem{G}Griffiths, P. A.: {\it Periods of integrals on algebraic manifolds. II. Local study of the period mapping}, Amer. J. Math. {\bf 90} (1968), 805--865.
\bibitem{La}Laterveer, R.: {\it Algebraic varieties with small Chow groups}, J. Math. Kyoto Univ. {\bf 38} (1998), no.~4, 673--694.
\bibitem{L} Lieberman, D.: {\it Intermediate Jacobians}, in {\it Algebraic geometry, Oslo 1970 (Proc. Fifth Nordic Summer-School in Math.)}, 125--139, Wolters-Noordhoff, Groningen.  \bibitem{Ma}Ma, S.: {\it Torsion 1-cycles and the coniveau spectral sequence}, Doc. Math. {\bf 22} (2017), 1501--1517. 
\bibitem{MS}Merkurev, A. S., Suslin, A.: {\it $K$-cohomology of Severi-Brauer varieties and the norm residue homomorphism}, Izv. Akad. Nauk SSSR Ser. Mat. {\bf 46} (1982), no.~5, 1011--1046, 1135--1136. 
\bibitem{M1}Murre, J. P.: {\it Un r\'esultat en th\'eorie des cycles alg\'ebriques de codimension deux}, C. R. Acad. Sci. Paris S\'er. I Math. {\bf 296} (1983), no.~23, 981--984. 
\bibitem{M2}Murre, J. P.: {\it Applications of algebraic $K$-theory to the theory of algebraic cycles}, in {\it Algebraic geometry, Sitges (Barcelona), 1983}, 216--261, Lecture Notes in Math., 1124, Springer, Berlin.
\bibitem{M3}Murre, J. P.: {\it Abel-Jacobi equivalence versus incidence equivalence for algebraic cycles of codimension two}, Topology {\bf 24} (1985), no.~3, 361--367.
\bibitem{P}Paranjape, K. H.  {\it Cohomological and cycle-theoretic connectivity}, Ann. of Math. (2) {\bf 139} (1994), no.~3, 641--660.
\bibitem{R}Roitman, A. A.: {\it The torsion of the group of $0$-cycles modulo rational equivalence}, Ann. of Math. (2) {\bf 111} (1980), no.~3, 553--569.
\bibitem{S}Saito, H.: {\it Abelian varieties attached to cycles of intermediate dimension}, Nagoya Math. J. {\bf 75} (1979), 95--119.
\bibitem{Sch}Schoen, C.: {\it On certain exterior product maps of Chow groups}, Math. Res. Lett. {\bf 7} (2000), no.~2-3, 177--194. 
\bibitem{SV}Soul\'e, C.; Voisin, C.: {\it Torsion cohomology classes and algebraic cycles on complex projective manifolds}, Adv. Math. {\bf 198} (2005), no.~1, 107--127. 
\bibitem{V1}Voisin, C.: {\it Some aspects of the Hodge conjecture}, Jpn. J. Math. {\bf 2} (2007), no.~2, 261--296.
\bibitem{V2}Voisin, C.: {\it Hodge theory and complex algebraic geometry. II}, translated from the French by Leila Schneps, reprint of the 2003 English edition, Cambridge Studies in Advanced Mathematics, 77, Cambridge University Press, Cambridge, 2007.
\bibitem{V3}Voisin, C. : {\it Degree $4$ unramified cohomology with finite coefficients and torsion codimension $3$ cycles}, in {\it Geometry and arithmetic}, 347--368, EMS Ser. Congr. Rep, Eur. Math. Soc., Z\"{u}rich.
\bibitem{V4}Voisin, C.: {\it Stable birational invariants and the L\"{u}roth problem}, in {\it Surveys in differential geometry 2016. Advances in geometry and mathematical physics}, 313--342, Surv. Differ. Geom., 21, Int. Press, Somerville, MA. 
\bibitem{W}Walker, M. E.: {\it The morphic Abel-Jacobi map}, Compos. Math. {\bf 143} (2007), no.~4, 909--944. 
\end{thebibliography}
\end{document}